\documentclass[12pt]{amsart}

\usepackage[cp1251]{inputenc}
\usepackage[T2A]{fontenc}
\usepackage{amsfonts}
\usepackage{amsbsy}
\usepackage{amssymb}
\usepackage[english]{babel}
\usepackage{hyperref}

\renewcommand{\abovecaptionskip}{0pt}
\renewcommand{\belowcaptionskip}{6pt}
\makeatletter

\renewcommand{\@makecaption}[2]{
\vspace{\abovecaptionskip}%
\sbox{\@tempboxa}{#1 #2}%
\global\@minipagefalse \hbox to \hsize {{\scshape \hfil #1 #2\hfil}}
\vspace{\belowcaptionskip}}

\makeatother

\newcommand{\al}{\alpha}

\newcommand{\rk}{\operatorname{rk}}

\newcommand{\SL}{\operatorname{SL}}

\newcommand{\Sp}{\operatorname{Sp}}

\newcommand{\diag}{\operatorname{diag}}

\DeclareMathOperator{\otimesZ}{\otimes\hspace{1pt}\rule[-3pt]{0pt}{0pt}_{\mathbb{Z}}}
 \DeclareMathOperator{\Supp}{Supp}

\newtheorem{theorem}{Theorem}

\newtheorem{proposition}{Proposition}
\newtheorem{lemma}{Lemma}
\newtheorem{corollary}{Corollary}

\newtheorem*{question*}{Question}

\theoremstyle{definition}

\newtheorem{dfn}{Definition}
\newtheorem*{dfn*}{Definition}
\newtheorem{example}{Example}

\theoremstyle{remark}

\newtheorem{remark}{Remark}

\begin{document}

\renewcommand{\proofname}{Proof}
\renewcommand{\abstractname}{Abstract}
\renewcommand{\refname}{References}
\renewcommand{\figurename}{Figure}
\renewcommand{\tablename}{Table}

\title[Harmonic analysis on spherical homogeneous spaces]{Harmonic analysis on spherical homogeneous spaces with solvable stabilizer}

\author{Roman Avdeev and Natalia Gorfinkel}

\thanks{Partially supported by Russian Foundation for Basic Research, grant no. 09-01-00648}

\address{Chair of Higher Algebra, Department of Mechanics and Mathematics,
Moscow State University, 1, Leninskie Gory, Moscow, 119992, Russia}

\email{suselr@yandex.ru, nataly.gorfinkel@gmail.com}




\begin{abstract}
For all spherical homogeneous spaces~$G/H$, where $G$ is a simply
connected semisimple algebraic group and $H$ a connected solvable
subgroup of~$G$, we compute the spectra of the representations of
$G$ on spaces of regular sections of homogeneous line bundles
over~$G/H$.
\end{abstract}

\maketitle

\sloppy

\section{Introduction} \label{introduction}

\subsection{} \label{intro}

Let $G$ be a connected semisimple complex algebraic group and let
$H$ be a closed subgroup of~$G$. One of the problems of harmonic
analysis on the homogeneous space $G/H$ is to find the spectrum of
the representation of $G$ on the space $\mathbb C[G/H]$ of regular
functions on~$G/H$. An important characteristic of this spectrum is
the so-called \textit{weight semigroup} $\Lambda_+(G/H)$. It
consists of dominant weights of $G$ such that the space $\mathbb
C[G/H]$, considered as a $G$-module, contains the irreducible
$G$-submodule with highest weight~$\lambda$.

The problem described above admits a natural generalization. Namely,
one may consider the problem of finding the spectra of the
representations of $G$ on spaces of regular sections of homogeneous
line bundles over~$G/H$. The whole collection of these spectra also
determines a semigroup $\widehat \Lambda_+(G/H)$ called the
\textit{extended weight semigroup} of the homogeneous space~$G/H$.
The exact definition of this semigroup will be given in
\S\,\ref{EWS}. The semigroup $\Lambda_+(G/H)$ is naturally
identified with a subsemigroup of $\widehat \Lambda_+(G/H)$ (see
\S\,\ref{EWS}).

A subgroup $H \subset G$ (resp. a homogeneous space $G/H$) is said
to be \textit{spherical} if a Borel subgroup $B \subset G$ has an
open orbit in~$G/H$. In~\cite{VK} the following criterion for $H$ to
be spherical was proved.

\begin{theorem}[{\rm \cite[Theorem~1]{VK}}]\label{sphericity}
\textup{(1)} A subgroup $H \subset G$ is spherical if and only if
for every homogeneous line bundle $L$ over $G/H$ the representation
of\, $G$ on the space of regular sections of $L$ is multiplicity
free.

\textup{(2)} In the case of quasi-affine~$G/H$, a subgroup $H
\subset G$ is spherical if and only if the representation of\, $G$
on the space $\mathbb C[G/H]$ of regular functions on $G/H$ is
multiplicity free.
\end{theorem}

It follows from this theorem that for a spherical homogeneous space
$G/H$ the semigroup $\widehat \Lambda_+(G/H)$ determines the
$G$-module structures on spaces of regular sections of all
homogeneous line bundles over~$G/H$ and $\Lambda_+(G/H)$ determines
the $G$-module structure on $\mathbb C[G/H]$. In this connection,
computation of the semigroups $\Lambda_+(G/H)$ and $\widehat
\Lambda_+(G/H)$ for spherical homogeneous spaces $G/H$ is of
interest.

At present, considerable advancements have been achieved in
computing weight semigroups for \textit{affine} spherical
homogeneous spaces $G/H$ (that is, with reductive $H$). Namely,
there is a description of the semigroups $\Lambda_+(G/H)$ for all
\textit{strictly irreducible} simply connected affine spherical
homogeneous spaces $G/H$ (see their definition, for instance,
in~\cite[\S\,1.4]{Avd_EWS}), among them are all simply connected
affine spherical homogeneous spaces of simple groups. For such
spaces $G/H$, in the case of simple $G$ the semigroups
$\Lambda_+(G/H)$ were computed in~\cite{Kr} and in the case of
non-simple $G$ the semigroups $\widehat \Lambda_+(G/H)$ were
computed in~\cite{Avd_EWS}. Using these results and simple
additional considerations, one can obtain a description of the
semigroups $\widehat \Lambda_+(G/H)$ for \textit{all} simply
connected affine spherical homogeneous spaces~$G/H$, however this
aspect has not yet been reflected in the literature.

\begin{remark}
As far as the authors are informed, in the case of non-simple $G$
the semigroups $\widehat \Lambda_+(G/H)$ for all simply connected
strictly irreducible affine spherical homogeneous spaces $G/H$ were
computed by Yu.\,V.~Dzyadyk as long ago as in~1985. However these
results have not been published.
\end{remark}

For non-affine spherical homogeneous spaces $G/H$, the state of the
art in the computation of the semigroups $\Lambda_+(G/H)$ or
$\widehat \Lambda_+(G/H)$ is much more complicated: the authors are
aware of only several particular cases of computing these
semigroups. For instance, it is not hard to describe the semigroup
$\widehat \Lambda_+(G/H)$ in the case where $H$ is an intermediate
subgroup between some parabolic subgroup of $G$ and its derived
subgroup. (Such subgroups $H$ are exactly the \textit{horospherical}
subgroups, that is, subgroups containing some maximal unipotent
subgroup of~$G$; see~\cite[\S\,2]{Kn} about that.) Another
particular case was examined in~\cite{Gor}, where the semigroups
$\widehat \Lambda_+(G/H)$ were computed in the case where $G$ is
simply connected and $H = T U'$ ($U$ is a maximal unipotent subgroup
of~$G$, $U'$ is its derived subgroup, $T$ is a maximal torus in~$G$
normalizing~$U$).

In the general case, for spherical homogeneous spaces $G/H$ it is
more convenient to compute the semigroup $\widehat \Lambda_+(G/H)$.
This is caused by the following reason: if $G$ is simply connected,
then $\widehat \Lambda_+(G/H)$ is free for any spherical subgroup $H
\subset G$ (see Theorem~\ref{semigroup_is_free} in
\S\,\ref{EWS_spherical}). In this case, to compute the semigroup
$\widehat \Lambda_+(G/H)$ it suffices to find its rank and present
the required number of its indecomposable elements. As far as the
semigroup $\Lambda_+(G/H)$ is concerned, it turns out to be free
much more rarely. Namely, for $G$ simply connected and $H$ connected
a known sufficient condition for this semigroup to be free is that
$H$ has no non-trivial characters (see~\cite[Proposition~2]{Pan90}).
We note that, generally speaking, this condition is quite
restrictive.

In the present paper we compute the semigroups $\widehat
\Lambda_+(G/H)$ for all spherical homogeneous spaces~$G/H$, where
$G$ is a simply connected semisimple group and $H$ is a connected
solvable subgroup of~$G$. The approach used for this purpose
combines ideas contained in~\cite{Gor} and~\cite{Avd_solv}. Namely,
we generalize the technique applied in~\cite{Gor} for computing the
semigroups $\widehat \Lambda_+(G/H)$ in the case $H = TU'$ (see
above) to the case of an arbitrary connected solvable spherical
subgroup $H \subset G$, making use of a structure theory of
connected solvable spherical subgroups in semisimple algebraic
groups developed in~\cite{Avd_solv}. The computation results are
expressed in terms of combinatorial data considered
in~\cite{Avd_solv}. The main result of this paper is
Theorem~\ref{main_theorem} (see \S\,\ref{theorem}).

\subsection{} \label{EWS}

Throughout this paper the ground field is the field $\mathbb C$ of
complex numbers, all topological terms relate to the Zarisky
topology, all groups are supposed to be algebraic and their
subgroups closed. Tangent algebras of groups denoted by capital
Latin letters are denoted by the corresponding small German letters.
For any group $L$ we denote by $\mathfrak X(L)$ the group of its
characters in additive notation.

In what follows, $G$ denotes a connected semisimple algebraic group.
In~$G$ we fix a Borel subgroup $B$ and a maximal torus $T$ contained
in~$B$. The maximal unipotent subgroup of $G$ contained in $B$ is
denoted by~$U$. The groups $\mathfrak X(B)$ and $\mathfrak X(T)$ are
identified by restricting characters from $B$ to~$T$. The set of
dominant weights of $G$ with respect to $B$ is denoted
by~$\Lambda_+(G)$; $\Lambda_+(G) \subset \mathfrak X(B)$. This
notation agrees with the notation $\Lambda_+(G/H)$ introduced in
\S\,\ref{intro} in the case $H=\{e\}$. For $\lambda \in
\Lambda_+(G)$ we denote by $V(\lambda)$ the irreducible $G$-module
with highest weight~$\lambda$ and by $v_\lambda$ a fixed
highest-weight vector in~$V(\lambda)$ (with respect to~$B$). For
every $\lambda \in \Lambda_+(G)$ the highest weight of the
irreducible $G$-module dual to $V(\lambda)$ is denoted
by~$\lambda^*$.

The actions of $G$ on itself by left translation ($(g,x)\mapsto gx$)
and right translation ($(g,x)\mapsto xg^{-1}$) induce
representations of $G$ on the space $\mathbb C[G]$ of regular
functions on $G$ by the formulas $(gf)(x) =f(g^{-1}x)$ and $(gf)(x)
=f(xg)$, respectively ($g,x \in G$, $f \in \mathbb C[G]$). For
brevity, we refer to these actions as the action \textit{on the
left} and \textit{on the right}, respectively (where $g,x \in G$ and
$f \in \mathbb C[G]$). For every subgroup $L\subset G$ we denote by
${}^L\mathbb C[G]$ (resp. $\mathbb C[G]^L$) the algebra of functions
in $\mathbb C[G]$ that are invariant under the action of $L$ on the
left (resp. on the right).

Let $H \subset G$ be an arbitrary subgroup. We recall
(see~\cite[Theorem~4]{Pop}) that the homogeneous line bundles over
$G/H$ are in one-to-one correspondence with the characters of~$H$.
Namely, a character $\chi \in \mathfrak X(H)$ corresponds to a
homogeneous line bundle $L(\chi) = (G \times \mathbb C_\chi)/H$ over
$G/H$, where $H$ acts on $G$ by right translation and on the space
$\mathbb C_\chi \simeq \mathbb C$ by means of the character~$\chi$.
The fiber of $L(\chi)$ over the point $eH$ is the line~$\mathbb
C_\chi$. For every $\chi \in \mathfrak X(H)$ there is a natural
$G$-equivariant isomorphism between the space $\Gamma(L(-\chi))$ of
regular sections of $L(-\chi)$ and the space
$$
V_{\chi} = \bigl\{f\in \mathbb{C}[G] \mid f(gh)=\chi(h)f(g)\ \
\forall\,g\in G, \ \forall\,h\in H\bigr\} \subset \mathbb C[G].
$$
Under this isomorphism, a function $f \in V_\chi$ corresponds to the
section $\gamma_f \in \Gamma(L(-\chi))$ given by the formula
$\gamma_f(gH) = [g,f(g)]$, where $[g,f(g)]$ is the class in
$L(-\chi)$ of the pair $(g,f(g))$. It is easy to see that the space
$V_0 \simeq \Gamma(L(0))$ corresponding to the character $\chi = 0$
is nothing else but the space $\mathbb C[G]^H = \mathbb C[G/H]$ of
regular functions on~$G/H$. We note that $\bigoplus \limits_{\chi
\in \mathfrak X(H)}V_\chi = \mathbb C[G]^{H_0}$, where the subgroup
$H_0 \subset H$ is the intersection of kernels of all characters
of~$H$.

Suppose that $\lambda \in \Lambda_+(G)$ and $\chi \in \mathfrak
X(H)$. A function $f \in \mathbb C[G]$ is said to be
\textit{$(B\times H)$-semi-invariant of weight $(\lambda,\chi)$} if
$f$ is semi-invariant of weight $\lambda$ under the action of $B$ on
the left and semi-invariant of weight $\chi$ under the action of $H$
on the right, that is, $f(b^{-1}gh)=\lambda(b)\chi(h)f(g)$ for all
$b \in B$, $g \in G$, $h \in H$. Every non-zero $(B\times
H)$-semi-invariant function of weight $(\lambda, \chi)$ is a highest
weight vector of some irreducible $G$-submodule in $V_\chi$ with
highest weight~$\lambda$, and vice versa. We denote by
$A(\lambda,\chi)$ the subspace in $\mathbb C[G]$ consisting of all
$(B\times H)$-semi-invariant functions of weight $(\lambda, \chi)$.
It is easy to see that the multiplicity with which the irreducible
$G$-module with highest weight $\lambda$ occurs in the $G$-module
$V_\chi$ equals $\dim A(\lambda,\chi)$.

\begin{dfn}
The \textit{extended weight semigroup} $\widehat \Lambda_+(G/H)$ of
a homogeneous space $G/H$ is the set of pairs $(\lambda, \chi)$
(where $\lambda \in \Lambda_+(G)$ and $\chi \in \mathfrak X(H)$)
such that $\dim A(\lambda,\chi) \geqslant 1$.
\end{dfn}

The inclusion $A(\lambda_1, \chi_1) A(\lambda_2, \chi_2) \subset
A(\lambda_1 + \lambda_2, \chi_1 + \chi_2)$ implies that the set
$\widehat \Lambda_+(G/H)$ is indeed a semigroup. We note that this
semigroup always contains the element $(0,0)$ since the space $V_0 =
\mathbb C[G/H]$ of regular functions on $G/H$ always contains the
constants. Thus $\widehat \Lambda_+(G/H)$ is a monoid.

Evidently, $\Lambda_+(G/H) \simeq \{(\lambda,\chi) \in \widehat
\Lambda_+(G/H) \mid \chi=0\} \subset \widehat \Lambda_+(G/H)$. In
particular, $\Lambda_+(G/H) \simeq \widehat \Lambda_+(G/H)$ whenever
$\mathfrak X(H) = 0$.

We recall that there is the following isomorphism of $(G \times
G)$-modules:
\begin{equation}\label{GG-module}
\mathbb C[G] \simeq \bigoplus \limits_{\lambda \in \Lambda_+(G)}
V(\lambda^*)\otimes V(\lambda);
\end{equation}
on the left-hand side, $G\times G$ acts on the left and on the
right, and in each summand on the right-hand side the left (resp.
right) factor of $G \times G$ acts on the left (resp. right) tensor
factor. We note that for fixed $\lambda \in \Lambda_+(G)$ the
($G\times G$)-module $V(\lambda^*) \otimes V(\lambda)$ is embedded
in $\mathbb C[G]$ as follows: the image of an element $u\otimes v
\in V(\lambda^*) \otimes V(\lambda)$ is the function whose value at
each point $g \in G$ equals $\langle u, gv \rangle$, where $\langle
\cdot\,, \cdot \rangle$ is the natural pairing between
$V(\lambda^*)$ and~$V(\lambda)$. Under isomorphism~(\ref{GG-module})
the subspace $A(\lambda^*, \chi) \subset \mathbb C[G]$ corresponds
to the subspace $v_{\lambda^*} \otimes V(\lambda)_{\chi}^{(H)}
\subset V(\lambda^*) \otimes V(\lambda)$, where
$V(\lambda)_{\chi}^{(H)} \subset V(\lambda)$ is the subspace of
$H$-semi-invariant vectors of weight~$\chi$. Hence $\dim
A(\lambda^*, \chi) = \dim V(\lambda)_{\chi}^{(H)}$ and $(\lambda^*,
\chi) \in \widehat \Lambda_+(G/H)$ if and only if
$V(\lambda)_{\chi}^{(H)} \ne 0$.

\subsection{} \label{EWS_spherical}

We now suppose that $H \subset G$ is a spherical subgroup. In this
case, Theorem~\ref{sphericity} implies that the condition
$(\lambda^*, \chi) \in \widehat \Lambda_+(G/H)$ is equivalent to
either of the two conditions $\dim A(\lambda^*,\chi) = 1$ and $\dim
V(\lambda)_{\chi}^{(H)} = 1$.

An important a priori information about $\widehat \Lambda_+(G/H)$ is
given by the following theorem.

\begin{theorem}\label{semigroup_is_free}
If $G$ is simply connected, then for every spherical subgroup ${H
\subset G}$ the semigroup $\widehat \Lambda_+(G/H)$ is free. More
precisely, $\widehat \Lambda_+(G/H)$ is isomorphic to the semigroup
$\mathcal D(G/H)$ of effective $B$-stable divisors in~$G/H$, which
is freely generated by the finite set of $B$-stable prime divisors
in~$G/H$.
\end{theorem}

In the case of a connected spherical subgroup $H \subset G$ the fact
that $\widehat \Lambda_+(G/H)$ is free can be proved by an argument
due to Panyushev, although in his paper~\cite{Pan90} he applied it
to a more particular situation, namely, to prove that the semigroup
$\Lambda_+(G/H)$ is free if $\mathfrak X(H) = 0$
(see~\cite[Proposition~2]{Pan90}). For an adaptation of this
argument to the situation under consideration
see~\cite[Theorem~1]{Avd_EWS}. The following proof of
Theorem~\ref{semigroup_is_free} in the general case was communicated
to the authors by D.\,A.~Timashev.

\begin{proof}[Proof of Theorem~\textup{\ref{semigroup_is_free}}]
Regard the map $\rho \colon \widehat \Lambda_+(G/H) \to \mathcal
D(G/H)$ defined as follows. To each element $(\lambda, \chi) \in
\widehat \Lambda_+(G/H)$ we assign the divisor of zeros of an
(arbitrary) non-zero section in the one-dimensional subspace
$A(\lambda, \chi) \subset V_\chi \simeq \Gamma(L(-\chi))$. It is
easy to see that this divisor is $B$-stable and the map $\rho$ is a
semigroup homomorphism. Conversely, let $D \in \mathcal D(G/H)$ be
an arbitrary divisor. Being a Weil divisor on~$G/H$, $D$ is locally
principal and hence determines a line bundle $L$ over $G/H$ together
with a section $s$ (uniquely determined up to proportionality) in
such a way that $D$ is the divisor of~$s$. As $D$ is effective, the
section $s$ is regular. Further, since $G$ is simply connected, its
action on $G/H$ `lifts' to an action on~$L$, that is, the bundle $L$
is homogeneous (see~\cite[Proposition~1 and Theorem~4]{Pop}). The
divisor $D$ is $B$-stable, hence $s$ is $B$-semi-invariant. It
follows that $\rho$ is a bijection and thereby an isomorphism.
\end{proof}

\begin{remark} \label{rem_irr}
In the case of simply connected $G$, the algebra $\mathbb C[G]$ is
factorial (see~\cite[Corollary from Proposition~1]{Pop}); therefore,
under the assumptions of Theorem~\ref{semigroup_is_free}, the
semigroup $\widehat \Lambda_+(G/H)$ is also isomorphic to the
semigroup of effective $(B \times H)$-stable divisors in~$G$. Under
this isomorphism an element $(\lambda, \chi) \in \widehat
\Lambda_+(G/H)$ corresponds to the divisor of zeros of some (any)
non-zero $(B \times H)$-semi-invariant function in $A(\lambda,
\chi)$. In the case of connected $H$ the semigroup of effective $(B
\times H)$-stable divisors in $G$ is freely generated by the finite
set of prime $(B \times H)$-stable divisors. In this situation, an
element $(\lambda, \chi) \in \widehat \Lambda_+(G/H)$ is
indecomposable if and only if the corresponding non-zero ${(B \times
H)}$-semi-invariant function in $A(\lambda, \chi)$ is irreducible
in~$\mathbb C[G]$.
\end{remark}

\subsection{Some notation and conventions. }~

$e$ is the identity element of any group

$|X|$ is the cardinality of a finite set~$X$

$V^*$ is the space of linear functions on a vector space~$V$

$\diag(a_1, \ldots, a_n)$ is the diagonal matrix of order $n$ with
elements $a_1, \ldots, a_n$ on the diagonal

For a group $L$, the notation $L = L_1 \rightthreetimes L_2$
signifies that $L$ decomposes into a semidirect product of subgroups
$L_1, L_2$, that is, $L = L_1L_2$, $L_1 \cap L_2 = \{e\}$, and $L_2$
is a normal subgroup of~$L$.

We identify the lattice $\mathfrak X(T)$ with a sublattice in
$\mathfrak t^*$ by associating each character $\mu \in \mathfrak
X(T)$ with its differential $d\mu \in \mathfrak t^*$.

In $\mathfrak X(T)$ (and thereby in $\mathfrak t^*$), we fix the
root system $\Delta$ with respect to $T$ and the set of positive
roots $\Delta_+ \subset \Delta$ with respect to~$B$. Let $\Pi =
\{\alpha_1, \ldots, \alpha_n\} \subset \Delta_+$ be the set of
simple roots and let $\omega_1, \ldots, \omega_n \in \mathfrak X(T)
\otimesZ \mathbb Q$ be the fundamental weights corresponding to the
simple roots $\alpha_1, \ldots, \alpha_n$, respectively. In the case
of simply connected $G$ one has $\omega_1, \ldots, \omega_n \in
\mathfrak X(T)$.

For every root $\alpha \in \sum\limits_{\gamma \in \Pi}k_\gamma
\gamma \in \Delta_+$, we put $\Supp \alpha = \{\gamma \mid k_\gamma
> 0\}$.

Let $W = N_G(T)/T$ be the Weyl group. In the space $\mathfrak X(T)
\otimesZ \mathbb Q$ we fix an inner product $(\cdot\,,\,\cdot)$
invariant under~$W$. For $\lambda, \mu \in \mathfrak X(T)$ (${\mu
\ne 0}$), we put $\langle \lambda |\, \mu \rangle =
\cfrac{2(\lambda,\mu)}{(\mu,\mu)}$.

Let $\Delta^\vee \subset \mathfrak t$ be the root system dual
to~$\Delta$. For every $\alpha \in \Delta$ we denote by $h_\alpha$
the corresponding element in~$\Delta^\vee$.

In each root subspace $\mathfrak g_\alpha \subset \mathfrak g$ we
choose a basis vector $e_\alpha$ in such a way that the condition
$[e_\alpha, e_{-\alpha}] = h_\alpha$ is fulfilled for every $\alpha
\in \Delta$.

\subsection*{Acknowledgements}

The authors are deeply grateful to E.\,B.~Vinberg and
D.\,A.~Timashev for reading the previous version of this paper and
valuable comments.

\section{Formulation of the main result} \label{result}

\subsection{} \label{theorem}

In order to state the main theorem, we need some facts from the
structure theory of connected solvable spherical subgroups in
semisimple algebraic groups; see~\cite{Avd_solv}.

Let $H \subset B$ be a connected solvable subgroup and let $N
\subset U$ be its unipotent radical. We say that $H$ is
\textit{standardly embedded} in~$B$ (with respect to~$T$) if the
subgroup $S = H \cap T \subset T$ is a maximal torus in~$H$.
Obviously, in this situation one has $H = S \rightthreetimes N$.
Every connected solvable subgroup in $G$ is conjugate by a suitable
element of $G$ to a subgroup standardly embedded in~$B$.

Suppose that a connected solvable spherical subgroup ${H \subset G}$
standardly embedded in~$B$ is fixed. As above, we put $S = H \cap T$
and $N = H \cap U$ so that $H = S \rightthreetimes N$. We identify
the groups $\mathfrak X(H)$ and $\mathfrak X(S)$ by restricting
characters from $H$ to~$S$. We denote by $\tau \colon \mathfrak
X(T)\to\mathfrak X(S)$ the character restriction map from $T$
to~$S$. Let $\Phi = \tau(\Delta_+) \subset \mathfrak X(S)$ be the
weight system of the action of $S$ on $\mathfrak u$ by means of the
adjoint representation of~$G$. One has $\mathfrak u = \bigoplus
\limits_{\varphi \in \Phi} \mathfrak u_\varphi$, where $\mathfrak
u_\varphi \subset \mathfrak u$ is the weight subspace of weight
$\varphi$ with respect to~$S$. Let $\mathfrak n = \bigoplus
\limits_{\varphi \in \Phi} \mathfrak n_\varphi$ be the decomposition
of $\mathfrak n$ into the direct sum of weight subspaces with
respect to~$S$. Here for every $\varphi \in \Phi$ one has $\mathfrak
n_\varphi \subset \mathfrak u_\varphi$. For every $\varphi \in \Phi$
let $c_\varphi$ denote the codimension of $\mathfrak n_\varphi$
in~$\mathfrak u_\varphi$.

In the notation introduced above, there is the following sphericity
criterion for a connected solvable subgroup in~$G$.

\begin{theorem}[{\rm \cite[Theorem~1]{Avd_solv}}]\label{solvable_spherical}
Let $H \subset G$ be a connected solvable subgroup standardly
embedded in~$B$. The following conditions are equivalent:

\textup{(1)}  $H$ is spherical in~$G$;

\textup{(2)} $c_\varphi \leqslant 1$ for every $\varphi \in \Phi$,
and all weights with $c_\varphi = 1$ are linearly independent
in~$\mathfrak X(S)$.
\end{theorem}

Later on, we assume $H \subset G$ to be a connected solvable
spherical subgroup standardly embedded in~$B$ and retain the
notation introduced above. We put $\Psi = {\{\alpha \in \Delta_+
\mid \mathfrak g_\alpha \not \subset \mathfrak n\}} \subset
\Delta_+$.

\begin{dfn}
The roots in $\Psi$ are called \textit{active}.
\end{dfn}

Let $\varphi_1, \ldots, \varphi_m$ denote all the weights $\varphi
\in \Phi$ with $c_\varphi = 1$. For $i = 1, \ldots, m$ we put
$\Psi_i = \{\alpha \in \Psi \mid \tau(\alpha) = \varphi_i\}$ and
$\mathfrak u_i = \bigoplus \limits_{\alpha \in \Psi_i} \mathfrak
g_\alpha$. It is obvious that $\Psi = \Psi_1 \cup \Psi_2 \cup \ldots
\cup \Psi_m$ and $\mathfrak u_i \subset \mathfrak u_{\varphi_i}$.

The set of active roots possesses the following property
(see~\cite[Lemma~4]{Avd_solv}): if $\alpha$ is an active root and
$\alpha = \beta + \gamma$ for some $\beta, \gamma \in \Delta_+$,
then exactly one of the two roots $\beta, \gamma$ is active.

\begin{dfn}
We say that an active root $\beta$ is \textit{subordinate} to an
active root~$\alpha$ if there is a root $\gamma \in \Delta_+$ such
that $\alpha = \beta + \gamma$.
\end{dfn}

\begin{proposition}[{\rm \cite[Proposition~3]{Avd_solv}}]
Let $\alpha$ be an active root. There is a unique simple root
${\pi(\alpha) \in \Supp \alpha}$ with the following property: if
$\alpha = \alpha_1 + \alpha_2$ for some $\alpha_1, \alpha_2 \in
\Delta_+$, then $\al_1$ \textup(resp.~$\alpha_2$\textup) is active
if and only if $\pi(\al) \notin \Supp \al_1$ \textup{(}resp.
$\pi(\al) \notin \Supp \al_2$\textup{)}.
\end{proposition}

Thus one has a map $\pi: \Psi \to \Pi$.

For every $j = 1, \ldots, m$, regard the set $\pi(\Psi_j)$. Let
$\alpha_{j_1}$, $\alpha_{j_2}$, $\ldots$, $\alpha_{j_r}$, where $r =
r(j) = |\pi(\Psi_j)|$, denote all roots contained in $\pi(\Psi_j)$.
We put $\lambda_j = \omega_{j_1} + \omega_{j_2} + \ldots +
\omega_{j_r}$.

We can now state the main result of the present paper.

\begin{theorem}\label{main_theorem}
If $G$ is simply connected, then the semigroup $\widehat
\Lambda_+(G/H)$ is freely generated by the elements $(\omega^*_i,
\tau(\omega_i))$, $i = 1, \ldots, n$, and the elements
$(\lambda^*_j, \tau(\lambda_j) - \varphi_j)$, $j = 1, \ldots, m$.
\end{theorem}

This theorem will be proved in~\S\,\ref{proof}.

\subsection{} \label{examples}

Let us present two examples of application of
Theorem~\ref{main_theorem}. In both cases, the sphericity of~$H$
follows from Theorem~\ref{solvable_spherical}.

\begin{example}{ \cite{Gor}}
Suppose that $G$ is simply connected and $H = TU'$. Then $S = T$, $N
= U'$, $\mathfrak n = \bigoplus \limits_{\alpha \in \Delta_+
\backslash \Pi} \mathfrak g_{\alpha}$, $\tau = \rm{id}$, $m = n$,
$\Psi = \Pi$, $\pi = \mathrm{id}$. The semigroup $\widehat
\Lambda_+(G/H)$ is freely generated by the elements $(\omega_i^*,
\omega_i)$, $(\omega_i^*, \omega_i - \alpha_i)$, where $i = 1,
\ldots, n$.
\end{example}

\begin{example}{}
Suppose that $G = \SL_4$ and the groups $B, U, T$ consist of all
upper-triangular, upper unitriangular, diagonal matrices,
respectively, contained in~$G$. For $t = \diag(t_1,t_2,t_3,t_4) \in
T$ and $k = 1, 2, 3$ we put $\alpha_k(t) = t_kt_{k+1}^{-1}$.
Consider a subgroup $H = S \rightthreetimes N \subset G$, where $S =
\{\diag(s_1, s_2, s_2^{-1}, s_1^{-1}) \mid s_1, s_2 \in \mathbb
C^\times\}$ and $\mathfrak n \subset \mathfrak u$ is the set of all
matrices of the form
$$\begin{pmatrix}0 & a & b & c\\ 0 & 0 & d & b\\
0 & 0 & 0 & -a \\ 0 & 0 & 0 & 0\end{pmatrix},
$$
where $a,b,c,d \in \mathbb C$ are arbitrary numbers. We note that
$H$ is a Borel subgroup of the group $\Sp_4 \subset G$ preserving
the skew-symmetric form with matrix
$$
\begin{pmatrix}
0 & 0 & 0 & 1\\
0 & 0 & 1 & 0\\
0 & -1 & 0 & 0\\
-1 & 0 & 0 &0
\end{pmatrix}.
$$
One has $n = 3$, $m = 2$, $\varphi_1 = \tau(\alpha_1) =
\tau(\alpha_3)$, $\varphi_2 = \tau(\alpha_1 + \alpha_2) =
\tau(\alpha_2 + \alpha_3)$, $\Psi_1 = \{\alpha_1, \alpha_3\}$,
$\Psi_2 = \{\alpha_1 + \alpha_2, \alpha_2 + \alpha_3\}$,
$\pi(\alpha_1) = \alpha_1$, $\pi(\alpha_3) = \alpha_3$,
$\pi(\alpha_1 + \alpha_2) = \pi(\alpha_2 + \alpha_3) = \alpha_2$.
For $s = \diag(s_1, s_2, s_2^{-1}, s_1^{-1}) \in S$ we put
$\chi_1(s) = s_1$, $\chi_2(s) = s_1s_2$. The semigroup $\widehat
\Lambda_+(G/H)$ is freely generated by the five elements $(\omega_3,
\tau(\omega_1)) = (\omega_3, \chi_1)$, $(\omega_2, \tau(\omega_2)) =
(\omega_2, \chi_2)$, $(\omega_1, \tau(\omega_3)) = (\omega_1,
\chi_1)$, $(\omega_1 + \omega_3, \tau(\omega_1 + \omega_3) -
\varphi_1) = (\omega_1 + \omega_3, \chi_2)$, $(\omega_2,
\tau(\omega_2) - \varphi_2) = (\omega_2, 0)$.
\end{example}

\section{Auxiliary results} \label{auxiliary results}

This section contains all facts needed in the proof of
Theorem~\ref{main_theorem}.

\subsection{} \label{approach_for_rank}

In this subsection we describe a general approach to computing the
rank of the semigroup $\widehat \Lambda_+(G/H)$ in the case of an
arbitrary connected subgroup $H \subset G$. This approach will be
applied in \S\,\ref{rank} to computing the rank of $\widehat
\Lambda_+(G/H)$ under the assumptions of Theorem~\ref{main_theorem}.

We first recall the following notion.

Let $L$ be a reductive group and let $B_L$ be a Borel subgroup
of~$L$. The \textit{rank} of the action $L : X$ of $L$ on an
irreducible variety $X$ is the rank of the lattice $\Lambda(X)
\subset \mathfrak X(B_L)$, where $\Lambda(X)$ consists of weights
$\mu \in \mathfrak X(B_L)$ such that the field $\mathbb C(X)$ of
rational functions on $X$ contains a non-zero $B_L$-semi-invariant
function of weight~$\mu$. We denote by $r_L(X)$ the rank of the
action $L:X$. By the rank of a homogeneous space $L/K$ we shall mean
the rank of the natural action $L : L/K$ by left translation.

Let $H$ be an arbitrary connected subgroup of~$G$. Regard the
homogeneous space $G/H_0$ (see the definition of the subgroup $H_0$
in~\S\,\ref{EWS}). It is quasi-affine since $H_0$ has no non-trivial
characters. There is a transitive action of the group $\widehat G =
G \times H/H_0$ on $G/H_0$, where $G$ acts on the left and $H/H_0$
acts on the right. Under this action, the stabilizer of the point
$eH_0$ is the subgroup $\widehat H = \{(h, hH_0) \mid h \in H\}
\subset \widehat G$, which is isomorphic to~$H$. Thus there is an
isomorphism of varieties $G/H_0 \simeq \widehat G/ \widehat H$. The
algebras $\mathbb C[G/H_0]$ and $\mathbb C[\widehat G/ \widehat H]$
are isomorphic as $\widehat G$-modules, hence there is a semigroup
isomorphism $\widehat \Lambda_+(G/H) \simeq \Lambda_+(\widehat G/
\widehat H)$. Since $\widehat G/ \widehat H$ is quasi-affine, the
lattice $\Lambda (\widehat G / \widehat H)$ is generated by the
semigroup $\Lambda_+(\widehat G/ \widehat H)$ (see, for
instance,~\cite[Proposition~5.14]{Tim}), therefore $\rk
\Lambda_+(\widehat G/ \widehat H) = \rk \Lambda(\widehat G/ \widehat
H)$. Hence $\rk \widehat \Lambda_+(G/H) = r_{\widehat G}(\widehat
G/\widehat H)$. The rank of the homogeneous space $\widehat G /
\widehat H$ can be computed by using the following general result of
Panyushev.

\begin{proposition}[{\rm \cite[Theorem~1.2(ii)]{Pan94}}] \label{formula_for_rank}
Let $L$ be a connected reductive group and let $P \subset L$ be a
parabolic subgroup with Levi decomposition $P = P_r \rightthreetimes
P_u$, where $P_r$ and $P_u$ are a maximal reductive subgroup and the
unipotent radical of~$P$, respectively. Let $K \subset L$ be a
connected subgroup with Levi decomposition $K = K_r \rightthreetimes
K_u$, where $K_r$ and $K_u$ are a maximal reductive subgroup and the
unipotent radical of~$K$, respectively, with $K_r \subset P_r$ and
$K_u \subset P_u$. Finally, let $F \subset K_r$ be a generic
stabilizer for the action $K_r: \mathfrak p_r / \mathfrak k_r$
\textup(the subgroup $F$ is reductive\textup). Then $r_L(L/K) =
r_{P_r}(P_r/K_r) + r_F(P_u/K_u)$.
\end{proposition}

\subsection{} \label{active_root_theory}

In this subsection we collect all necessary results of active root
theory.

\begin{proposition}\label{Psi's}
Suppose that $1 \leqslant i,j \leqslant m$ and different roots
$\alpha \in \Psi_i$, $\beta \in \Psi_j$ are such that $\gamma =
\beta - \alpha \in \Delta_+$. Then

\textup{(a)~\cite[Proposition~1]{Avd_solv}} $\Psi_i + \gamma \subset
\Psi_j$;

\textup{(b)~\cite[Proposition~10(c)]{Avd_solv}} if $|\Psi_i|
\geqslant 2$, then $\gamma$ is a unique positive root with the
property $\Psi_i + \gamma \subset \Psi_j$.
\end{proposition}

\begin{corollary}\label{root_difference}
Suppose that $1 \leqslant i \leqslant m$ and roots $\alpha, \beta
\in \Psi_i$ are different. Then $\alpha - \beta \notin \Delta$.
\end{corollary}

\begin{proof}
Assume that $\gamma = \alpha - \beta \in \Delta$. We may also assume
that $\gamma \in \Delta_+$. Then in view of
Proposition~\ref{Psi's}(a) one has $\Psi_i + \gamma \subset \Psi_i$,
which is impossible.
\end{proof}

For every $i = 1, \ldots, m$ the subspace $\mathfrak n \cap
\mathfrak u_i$ is determined inside $\mathfrak u_i$ by the vanishing
of a linear function (uniquely determined up to proportionality),
which will be denoted by~$\xi_i$. If $\alpha \in \Psi_i$ for some $i
\in \{1, \ldots, m\}$, then the restriction of $\xi_i$ to $\mathfrak
g_\alpha$ is not zero.

\begin{proposition}[{\rm \cite[Proposition~2]{Avd_solv}}] \label{xi's}
If\, $\Psi_i + \gamma \subset \Psi_j$ for some $i,j \in \{1, \ldots,
m\}$ and $\gamma \in \Delta_+$, then there exists a number $c \ne 0$
such that $\xi_i(x) = c \xi_j ([x, e_\gamma])$ for all $x \in
\mathfrak u_i$.
\end{proposition}

\begin{proposition}[{\rm \cite[Corollary~11]{Avd_solv}}] \label{al_sim_be}
Let active roots $\alpha, \beta$ be such that $\tau(\alpha) =
\tau(\beta)$. Then either $\pi(\alpha) = \pi(\beta)$ or none of the
roots $\pi(\alpha), \pi(\beta)$ is contained in $\Supp \alpha \cap
\Supp \beta$.
\end{proposition}

\section{Proof of the main theorem} \label{proof}

In view of Theorem~\ref{semigroup_is_free} the proof of
Theorem~\ref{main_theorem} involves three stages. First, in
\S\,\ref{rank} we compute the rank of the semigroup $\widehat
\Lambda_+(G/H)$. Then in \S\,\ref{contained} we show that all the
elements mentioned in the statement of Theorem~\ref{main_theorem}
are indeed contained in this semigroup. At last, in
\S\,\ref{indecomposable} we prove that these elements are
indecomposable in $\widehat \Lambda_+(G/H)$.

In this section we preserve the notation introduced in
\S\,\ref{theorem}.

\subsection{} \label{rank}

In this subsection we prove the following proposition.

\begin{proposition} \label{prop_rank}
Under the assumptions of Theorem~\textup{\ref{main_theorem}}, the
rank of $\widehat \Lambda_+(G/H)$ equals $n + m$.
\end{proposition}

Below we present two different proofs of this proposition. The first
one uses the general approach for computing the rank of extended
weight semigroups, which was described in
\S\,\ref{approach_for_rank}. The second proof is direct and uses a
geometric argument.

\begin{proof}[Proof~\textup{1}]
The natural epimorphism $H \to H/H_0$ maps the subgroup $S \subset
H$ isomorphically onto~$H/H_0$. For this reason, further we identify
$S$ and~$H/H_0$. It was shown in \S\,\ref{approach_for_rank} that
$\rk \widehat \Lambda_+(G/H) = r_{\widehat G} (\widehat G / \widehat
H)$. To compute $r_{\widehat G}(\widehat G / \widehat H)$ we apply
Proposition~\ref{formula_for_rank} with $L = \widehat G = G \times
S$ and $K = \widehat H$. Let $P = B \times S$. Then $P_r = T \times
S$, $P_u = U$; $K_r = \widehat S$, $K_u = N$, where the subgroup
$\widehat S \subset \widehat H$ is isomorphic to $S$ and embedded
diagonally in the subgroup $S \times S \subset H \times S$. Since
$\widehat S$ acts trivially on $\mathfrak p_r / \mathfrak k_r \simeq
(\mathfrak t \oplus \mathfrak s) / \widehat{\mathfrak s}$, one has
$F = \widehat S$. Therefore $r_{\widehat G}(\widehat G / \widehat H)
= r_{T \times S}((T \times S) / \widehat S) + r_{\widehat S}(U/N)$.
It is easy to see that the first summand in the last sum is equal to
$\rk T = n$. Let us find the value of the second summand. Clearly,
this value coincides with $r_S(U/N)$. In view
of~\cite[Lemma~1.4]{Mon} there is an $S$-equivariant isomorphism
$U/N \simeq \mathfrak u / \mathfrak n$. By
Theorem~\ref{solvable_spherical}, $\mathfrak u/\mathfrak n$ is
isomorphic as an $S$-module to the direct sum $\mathbb C_{\varphi_1}
\oplus \ldots \oplus \mathbb C_{\varphi_m}$, where $\mathbb
C_{\varphi_i} \simeq \mathbb C$ is the subspace of weight
$\varphi_i$ with respect to~$S$ and the weights $\varphi_1, \ldots,
\varphi_m$ are linearly independent. This implies that $r_S(U/N) =
m$.
\end{proof}

\begin{proof}[Proof~\textup{2}]
Since $H$ is connected, in view of Remark~\ref{rem_irr} the rank of
$\widehat \Lambda_+(G/H)$ equals the number of $(B \times H)$-stable
prime divisors in~$G$. Regard the Bruhat decomposition of~$G$:
$$
G = \bigsqcup \limits_{\sigma \in W} B \sigma B,
$$
where the union is disjoint. Among the subsets $B \sigma B \subset
G$, each of them being $(B \times B)$-stable, there is the open
subset $G_0 = B \sigma_0 B \subset G$ (referred to as the open
cell), where $\sigma_0$ is the longest element in~$W$. It is well
known that the complement of $G_0$ in~$G$ is the union of exactly
$n$ prime divisors, which are $(B\times B)$-stable (and so $(B
\times H)$-stable). Hence to complete the proof it remains to show
that the number of $(B \times H)$-stable prime divisors in~$G_0$
equals~$m$. In view of the isomorphism $G_0 \simeq B\sigma_0 \times
U$ we obtain that the $(B \times H)$-stable divisors in~$G_0$ are in
one-to-one correspondence with the $H$-stable divisors in~$U$, where
the action of $H$ on $U$ is given by the formula $(sv, u) \mapsto
suv^{-1}s^{-1}$ ($s \in S$, $v \in N$, $u \in U$). In turn, the
$H$-stable divisors in~$U$ are in one-to-one correspondence with the
$S$-stable divisors in~$U/N$ or, in view of an $S$-equivariant
isomorphism $U/N \simeq \mathfrak u / \mathfrak n$
(see~\cite[Lemma~1.4]{Mon}), with the $S$-stable divisors
in~$\mathfrak u / \mathfrak n$. Taking into account an $S$-module
isomorphism $\mathfrak u/\mathfrak n \simeq \mathbb C_{\varphi_1}
\oplus \ldots \oplus \mathbb C_{\varphi_m}$ (see
Theorem~\ref{solvable_spherical}), we see that the $S$-stable prime
divisors in $\mathfrak u / \mathfrak n$ are exactly the divisors of
zeros of $S$-semi-invariant linear functions on~$\mathfrak u /
\mathfrak n$. Evidently, up to proportionality there are exactly $m$
such linear functions.
\end{proof}

\subsection{} \label{contained}

In this subsection we show that the elements listed in the statement
of Theorem~\ref{main_theorem} are contained in the
semigroup~$\widehat \Lambda_+(G/H)$.

As was mentioned earlier (see \S\,\ref{EWS}), the condition
$(\lambda^*, \chi) \in \widehat \Lambda_+(G/H)$ is equivalent to the
condition $V(\lambda)_{\chi}^{(H)} \ne 0$. For the elements
$(\omega_i^*, \tau(\omega_i))$, $i = 1, \ldots, n$, it is quite easy
to point out a non-zero vector in
$V(\omega_i)_{\tau(\omega_i)}^{(H)}$: this is the highest weight
vector $v_{\omega_i} \in V(\omega_i)$. Now it remains to prove that
$(\lambda_j^*, \tau(\lambda_j) - \varphi_j) \in \widehat
\Lambda_+(G/H)$ or, equivalently,
$V(\lambda_j)^{(H)}_{\tau(\lambda_j) - \varphi_j} \ne 0$ for all $j
= 1, \ldots, m$. This is the objective of the rest of the
subsection.

Below we shall need the following simple lemma.

\begin{lemma} \label{highest_vector&two_roots}
Suppose that $\mu \in \Lambda_+(G)$ and $\alpha, \beta \in
\Delta_+$. Then:

\textup{(a)} $e_{\alpha} (e_{-\beta} v_\mu) = [e_{\alpha},
e_{-\beta}] v_\mu$;

\textup{(b)} if $[e_{\alpha}, e_{-\beta}] v_\mu \ne 0$, then $\beta
- \alpha \in \Delta_+$.
\end{lemma}

We fix $j \in \{1, \ldots, m\}$ and denote by $\beta_1, \ldots,
\beta_p$ all roots in~$\Psi_j$. The linear function $\xi_j \in
\mathfrak u_j^*$ (see \S\,\ref{active_root_theory}) is determined by
the set of non-zero numbers $a_1, \ldots, a_p$ as follows: an
element $x_1e_{\beta_1} + \ldots + x_pe_{\beta_p} \in \mathfrak u_j$
lies in $\mathfrak n$ if and only if $a_1x_1 + \ldots + a_px_p = 0$.

Suppose that $\pi(\Psi_j) =\{\alpha_{j_1}, \ldots, \alpha_{j_r}\}$
and $\lambda_j = \omega_{j_1} + \ldots + \omega_{j_r} \in
\Lambda_+(G)$ (see \S\,\ref{theorem}). We note that for every $k =
1, \ldots, p$ one has $\lambda_j(h_{\beta_k}) = \langle \lambda_j
|\, \beta_k \rangle > 0$ in view of the condition $\pi(\beta_k) \in
\Supp \beta_k$, whence $e_{-\beta_k} v_{\lambda_j} \ne 0$. We set
$$
f_j = \cfrac{a_1}{\lambda_j(h_{\beta_1})} e_{-\beta_1} + \ldots +
\cfrac{a_p}{\lambda_j(h_{\beta_p})} e_{-\beta_p} \in \mathfrak g
$$
and regard the (non-zero) vector $w_j = f_j v_{\lambda_j} \in
V(\lambda_j)$.

Our goal is to show that $w_j \in
V(\lambda_j)^{(H)}_{\tau(\lambda_j) - \varphi_j}$. It is easy to see
that the vector $w_j$ is $S$-semi-invariant of weight
$\tau(\lambda_j) - \varphi_j$. Let us prove that $w_j$ is
$H$-semi-invariant. For this purpose, it suffices to prove the
following proposition.

\begin{proposition} \label{annul}
The vector $w_j$ is annihilated by the algebra~$\mathfrak n$.
\end{proposition}

\begin{proof}[Proof \textup{consists of three steps.}]~

\textit{Step}~1. Let us prove that $w_j$ is annihilated by
$\mathfrak n \cap \mathfrak u_j$. For this purpose, we take an
arbitrary element $x = x_1e_{\beta_1} + \ldots + x_pe_{\beta_p} \in
\mathfrak n \cap \mathfrak u_j$ and show that $x w_j = 0$. In view
of Lemma~\ref{highest_vector&two_roots}(a) one has
$$
x w_j = \left(\sum \limits_{k=1}^p x_ke_{\beta_k}\right) \left( \sum
\limits_{l=1}^p \cfrac{a_l}{\lambda_j(h_{\beta_l})}
e_{-\beta_l}\right) v_{\lambda_j} = \sum \limits_{k=1}^p \sum
\limits_{l=1}^p \left(\cfrac{a_lx_k}{\lambda_j(h_{\beta_l})}
[e_{\beta_k}, e_{-\beta_l}] v_{\lambda_j} \right).
$$
By Corollary~\ref{root_difference}, for all $k \ne l$ one has
$\beta_k - \beta_l \notin \Delta$, whence $[e_{\beta_k},
e_{-\beta_l}] = 0$. Therefore
$$
x w_j = \sum \limits_{k=1}^p
\left(\cfrac{a_kx_k}{\lambda_j(h_{\beta_k})} [e_{\beta_k},
e_{-\beta_k}] v_{\lambda_j} \right) = \sum \limits_{k=1}^p
\left(\cfrac{a_kx_k}{\lambda_j(h_{\beta_k})}h_{\beta_k}
v_{\lambda_j} \right) = \left(\sum \limits_{k=1}^p a_k x_k \right)
v_{\lambda_j} = 0,
$$
where the relations $h_{\beta_k} v_{\lambda_j} =
\lambda_j(h_{\beta_k})v_{\lambda_j}$ and $a_1x_1 + \ldots + a_px_p =
0$ are taken into account.

\textit{Step}~2. Let us prove that $w_j$ is annihilated by
$\mathfrak n \cap \mathfrak u_i$ for $i\ne j$. Suppose that $\Psi_i
= \{\gamma_1, \ldots, \gamma_q\}$. If $|\Psi_i| = 1$, then
$\mathfrak n \cap \mathfrak u_i = \{0\}$ and there is nothing to
prove. So further we assume $|\Psi_i| \geqslant 2$. The linear
function $\xi_i \in \mathfrak u_i^*$ is determined by a set of
non-zero numbers $b_1, \ldots, b_q$ so that an element
$y_1e_{\gamma_1} + \ldots + y_qe_{\gamma_q}$ lies in $\mathfrak n$
if and only if $b_1y_1 + \ldots + b_qy_q = 0$. Suppose $y =
y_1e_{\gamma_1} + \ldots + y_qe_{\gamma_q} \in \mathfrak n \cap
\mathfrak u_i$. Let us show that $y w_j = 0$. In view of
Lemma~\ref{highest_vector&two_roots}(a) we have
$$
y w_j = \left(\sum \limits_{k=1}^q y_k e_{\gamma_k}\right) \left(
\sum \limits_{l=1}^p \cfrac{a_l}{\lambda_j(h_{\beta_l})}
e_{-\beta_l}\right) v_{\lambda_j} = \sum \limits_{k=1}^q \sum
\limits_{l=1}^p \left(\cfrac{a_l y_k}{\lambda_j(h_{\beta_l})}
[e_{\gamma_k}, e_{-\beta_l}] v_{\lambda_j} \right).
$$

Assume that for some $k \in \{1, \ldots, q\}$ and $l \in \{1,
\ldots, p\}$ the element $[e_{\gamma_k}, e_{-\beta_l}]$ acts
non-trivially on~$v_{\lambda_j}$. Then by
lemma~\ref{highest_vector&two_roots}(b) we get $\gamma_k - \beta_l =
-\delta$ for some ${\delta \in \Delta_+}$. In view of
Proposition~\ref{Psi's}(a) one has $\Psi_i + \delta \subset \Psi_j$.
Since $|\Psi_i| \geqslant 2$, by Proposition~\ref{Psi's}(b) the root
$\delta$ is the only positive root with property $\Psi_i + \delta
\subset \Psi_j$. Renumbering the roots in~$\Psi_j$, without loss of
generality we may assume that $\gamma_1 + \delta = \beta_1, \ldots,
\gamma_q + \delta = \beta_q$. Then
$$
y w_j = \sum \limits_{k=1}^q \left(\cfrac{a_k
y_k}{\lambda_j(h_{\beta_k})} [e_{\gamma_k}, e_{-\beta_k}]
v_{\lambda_j} \right) = \left(\sum \limits_{k=1}^q \cfrac{a_k d_k
y_k}{\lambda_j(h_{\beta_k})}\right) e_{-\delta} v_{\lambda_j},
$$
where for $k = 1, \ldots, q$ the numbers $d_k$ are such that
$[e_{\gamma_k}, e_{-\beta_k}] = d_k e_{-\delta}$. Further, for $k =
1, \ldots, q$ one has $[e_{\gamma_k}, e_\delta] = c_k e_{\beta_k}$,
where $c_k \ne 0$. In view of Proposition~\ref{xi's} we may assume
that $b_k = a_k c_k$ for all $k = 1, \ldots, q$.

To complete the proof, it suffices to show that
$$
\cfrac{a_1 d_1 y_1}{\lambda_j(h_{\beta_1})} + \ldots + \cfrac{a_q
d_q y_q}{\lambda_j(h_{\beta_q})} = 0.
$$
Substituting $a_k = \cfrac{b_k}{c_k}$ into this formula for $k = 1,
\ldots, q$ and taking into account the condition $b_1 y_1 + \ldots +
b_q y_q = 0$, we find that now it suffices to prove the relation
$$
\cfrac{d_1}{c_1 \lambda_j(h_{\beta_1})} = \ldots = \cfrac{d_q}{c_q
\lambda_j(h_{\beta_q})}.
$$
To this end, for fixed $k \in \{1, \ldots, q\}$ we regard the three
elements $e_{\gamma_k}, e_\delta, e_{-\beta_k}$ and write down the
Jacobi identity for them:
$$
[[e_{\gamma_k}, e_\delta], e_{-\beta_k}] + [[e_{-\beta_k},
e_{\gamma_k}], e_\delta] + [[e_\delta, e_{-\beta_k}], e_{\gamma_k}]
= 0,
$$
or
$$
c_k h_{\beta_k}+d_k h_\delta + A_k h_{\gamma_k} = 0,
$$
where $A_k$ is some number. Acting on $v_{\lambda_j}$ by both sides
of this equality, we obtain
$$
c_k \lambda_j(h_{\beta_m}) + d_k \lambda_j(h_{\delta}) + A_k
\lambda_j(h_{\gamma_k}) = 0.
$$
Since $\gamma_k$ is an active root subordinate to the active
root~$\beta_k$, it follows that the simple root $\pi(\beta_k) =
\alpha_{j_l}$ is not contained in the set $\Supp \gamma_k$, whence
$\omega_{j_l}(h_{\gamma_k}) = 0$. Further, in view of condition
$\Supp \gamma_m \subset \Supp \beta_m$ and
Proposition~\ref{al_sim_be} we obtain ${\pi(\Psi_j) \cap \Supp
\gamma_k = \varnothing}$. Hence $\omega_{j_1}(h_{\gamma_k}) = \ldots
= \omega_{j_r} (h_{\gamma_k}) = 0$ and $\lambda_j(h_{\gamma_k}) =
0$. Thus, $c_k \lambda_j(h_{\beta_k}) + {d_k \lambda_j(h_{\delta}) =
0}$. We note that $\lambda_j(h_{\delta}) = \langle \lambda_j |\,
\delta \rangle = \langle\omega_{j_l} |\, \delta\rangle > 0$ because
$\alpha_{j_l} \in \Supp \delta$. Hence the quantity
$$
\cfrac{d_k}{c_k \lambda_j(h_{\beta_k})} =
-\cfrac1{\lambda_j(h_{\delta})}
$$
is independent of~$k$, as required.

\textit{Step}~3. Let us prove that $w_j$ is annihilated by
$\mathfrak g_\gamma$ for every $\gamma \in \Delta_+ \backslash
\Psi$. In view of Lemma~\ref{highest_vector&two_roots}(a) we have
$$
e_{\gamma} w_j = e_{\gamma} \left( \sum \limits_{k=1}^p
\cfrac{a_k}{\lambda_j(h_{\beta_k})} e_{-\beta_k}\right)
v_{\lambda_j} = \sum \limits_{k=1}^p
\left(\cfrac{a_k}{\lambda_j(h_{\beta_k})} [e_{\gamma}, e_{-\beta_k}]
v_{\lambda_j} \right).
$$
Let us show that each summand of the latter sum equals zero. Assume
the converse: $[e_{\gamma}, e_{-\beta_k}] v_{\lambda_j} \ne 0$ for
some $k \in \{1, \ldots, p\}$. Then by
Lemma~\ref{highest_vector&two_roots}(b) we obtain $\delta = {\beta_k
- \gamma \in \Delta_+}$. Hence $e_{-\delta} v_{\lambda_j} \ne 0$ and
$\langle \lambda_j |\, \delta \rangle
>0$. On the other hand, from the conditions $\beta_k = \gamma + \delta$
and $\gamma \notin \Psi$ it follows that $\delta$ is an active root
subordinate to the active root~$\beta_k$. This means that
$\pi(\beta_k) \notin \Supp \delta$. Then by
Proposition~\ref{al_sim_be} one has ${\pi(\Psi_i) \cap \Supp \delta
= \varnothing}$, whence $\langle \lambda_j |\, \delta \rangle = 0$.
This contradiction completes \textit{Step}~3 and the proof of the
proposition.
\end{proof}

\begin{remark}
If $p = 1$ (that is, $|\Psi_j| = 1$), then the proof of
Proposition~\ref{annul} considerably simplifies. Namely, in this
case there is nothing to prove at \textit{Step}~1 (since $\mathfrak
n \cap \mathfrak u_j = 0$) whereas \textit{Step}~2 looses its
substantial part: for $i \ne j$ and $|\Psi_i| \geqslant 2$ the
condition $\Psi_i + \gamma \subset \Psi_j$ can hold for no root
$\gamma \in \Delta_+$. We also note that for $p = 1$ the vector
$w_j$ is proportional to~$e_{-\beta_1} v_{\lambda_j}$.
\end{remark}

\subsection{} \label{indecomposable}

In this subsection we prove that the elements of the semigroup
$\widehat \Lambda_+(G/H)$ listed in the statement of
Theorem~\ref{main_theorem} are indecomposable in this semigroup.
Since $H$ is connected, in view of Remark~\ref{rem_irr} it suffices
to prove that the non-zero ${(B \times H)}$-semi-invariant regular
functions on~$G$ corresponding to the above-mentioned elements are
irreducible in~$\mathbb C[G]$.

For $i = 1, \ldots, n$ we denote by $P_i$ the function in $\mathbb
C[G]$ corresponding to the element $v_{\omega_i^*} \otimes
v_{\omega_i} \in V(\omega_i^*) \otimes V(\omega_i)$ under
isomorphism~(\ref{GG-module}). One has $P_i(g) = \langle
v_{\omega_i^*}, g v_{\omega_i} \rangle$ for all $g \in G$. The
function $P_i$ is $(B \times H)$-semi-invariant of weight
$(\omega_i^*, \tau(\omega_i))$. Similarly, for $j = 1, \ldots, m$ we
denote by $Q_j$ the function in $\mathbb C[G]$ corresponding to the
element $v_{\lambda_j^*} \otimes w_j \in V(\lambda_j^*) \otimes
V(\lambda_j)$ under isomorphism~(\ref{GG-module}). We have $Q_j(g) =
\langle v_{\lambda_j^*}, g w_j \rangle$ for all~$g \in G$. The
function $Q_i$ is $(B \times H)$-semi-invariant of weight
$(\lambda_j^*, \tau(\lambda_j) - \varphi_j)$.

\begin{lemma}\label{irreducible}
Suppose that $i \in \{1, \ldots, n\}$ and $v \in V(\omega_i)
\backslash \{0\}$. Then the function $f \in \mathbb C[G]$
corresponding to the element $v_{\omega_i^*} \otimes v \in
V(\omega_i^*) \otimes V(\omega_i)$ under
isomorphism~\textup{(\ref{GG-module})} is irreducible in~$\mathbb
C[G]$.
\end{lemma}

\begin{proof}
Regard the algebra ${}^U\mathbb C[G]$. In view of
isomorphism~(\ref{GG-module}) we have
$$
{}^U \mathbb C[G] \simeq \bigoplus \limits_{\lambda \in
\Lambda_+(G)} v_{\lambda^*} \otimes V(\lambda),
$$
where the component $v_{\lambda^*} \otimes V(\lambda)$ is regarded
as a subspace in $V(\lambda^*) \otimes V(\lambda)$. The action on
the left of the torus $T$ determines a grading on~${}^U \mathbb
C[G]$ by elements of $\Lambda_+(G)$ in such a way that the component
of weight $\lambda^*$ corresponds to the subspace $v_{\lambda^*}
\otimes V(\lambda)$. The hypothesis of the lemma implies that $f$ is
contained in the component of weight $\omega_i^*$ of ${}^U \mathbb
C[G]$. Since the component of weight $0$ contains only the constants
and the weight $\omega_i^*$ admits no non-trivial expression as a
sum of two elements in~$\Lambda_+(G)$, it follows that the function
$f$ is irreducible in ${}^U \mathbb C[G]$. It remains to notice
that, for a function in ${}^U \mathbb C[G]$, irreducibility in ${}^U
\mathbb C[G]$ is equivalent to irreducibility in~$\mathbb C[G]$,
since the group $U$ is connected and has no non-trivial characters
(see~\cite[Theorem~3.17]{VP}).
\end{proof}

\begin{corollary}
For all $i = 1, \ldots, n$ the function $P_i$ is irreducible
in~$\mathbb C[G]$.
\end{corollary}

We note that for every $i = 1, \ldots, n$ the function $P_i$ is not
only $(B\times H)$-semi-invariant, but also $(B\times
B)$-semi-invariant. Therefore the corresponding $(B \times
H)$-stable prime divisor $D_i \subset G$ is also $(B \times
B)$-stable. As was already mentioned in \S\,\ref{rank} (see
\textit{Proof}~2 of Proposition~\ref{prop_rank}), the complement in
$G$ of the open Bruhat cell $G_0 = B \sigma_0 B$ is the union of
exactly $n$ prime divisors, which are $(B \times B)$-stable. It is
clear that these divisors are nothing else but $D_1, \ldots, D_n$.

To complete the proof of Theorem~\ref{main_theorem} it remains to
prove that the function $Q_j$ is irreducible in $\mathbb C[G]$ for
$j = 1, \ldots, m$.

\begin{proposition}\label{non-divisible}
For all $i = 1, \ldots, n$ and $j = 1, \ldots, m$ the function $Q_j$
is not divisible by~$P_i$.
\end{proposition}

\begin{proof}
We fix $i \in \{1, \ldots, n\}$, $j \in \{1, \ldots, m\}$ and show
that $Q_j$ is not divisible by~$P_i$. The function $Q_j/P_i$ is a
$(B\times H)$-semi-invariant rational function on $G$ of weight
$({\lambda_j^* - \omega_i^*}, {\tau(\lambda_j - \omega_i) -
\varphi_j)}$. If $i \notin \{j_1, \ldots, j_r\}$, then the weight
$\lambda_j^* - \omega_i^*$ is not dominant, therefore the function
$Q_j/P_i$ is not regular. Taking this into account, further we
suppose that $i \in \{j_1, \ldots, j_r\}$. Moreover, without loss of
generality we may assume $i = j_1$. We recall that $w_j = f_j
v_{\lambda_j}$, where the element $f_j$ is defined in
\S\,\ref{contained}. The function corresponding to the element
$v_{\lambda_j^*} \otimes v_{\lambda_j}$ under
isomorphism~(\ref{GG-module}) is $(B\times H)$-semi-invariant of
weight $(\lambda_j^*, \lambda_j)$. In view of the condition $\dim
A(\lambda_j^*, \lambda_j) \leqslant 1$ we obtain that this function
coincides with the function $c P_i P_{j_2} \ldots P_{j_r}$ for some
$c \ne 0$. Without loss of generality we assume $c=1$. Then one has
$Q_j = f_j (P_i P_{j_2} \ldots P_{j_r})$, where the action on
$\mathbb C[G]$ of the Lie algebra $\mathfrak g$ is induced by the
action of $G$ on the right. The element $f_j$ acts on the product
$P_i P_{j_2} \ldots P_{j_r}$ as a derivation, therefore $Q_j = (f_j
P_i)P_{j_2} \ldots P_{j_r} +P_i \cdot f_j(P_{j_2} \ldots P_{j_r})$.
Hence the divisibility of $Q_j$ by $P_i$ is equivalent to the
divisibility of $f_jP_i$ by~$P_i$. Under
isomorphism~(\ref{GG-module}), the function $f_jP_i$ corresponds to
the vector $v_{\omega_i^*} \otimes (f_j v_{\omega_i}) \in
V(\omega_i^*) \otimes V(\omega_i)$. Let us show that this vector is
non-zero, that is, $f_j v_{\omega_i} \ne 0$. For that, it suffices
to find a root $\beta \in \Psi_j$ such that $e_{-\beta}v_{\omega_i}
\ne 0$. The last condition is equivalent to the inequality $\langle
\omega_i |\, \beta \rangle
>0$, which surely holds for any root $\beta \in \Psi_j$ with
$\pi(\beta) = \alpha_i$. Thus, the vector $f_j v_{\omega_i} \in
V(\omega_i)$ is non-zero. It is easy to see that this vector is not
proportional to~$v_{\omega_i}$. In view of Lemma~\ref{irreducible}
this implies that $f_jP_i$ is irreducible in $\mathbb C[G]$ and is
not proportional to~$P_i$. Hence $f_jP_i$ is not divisible by~$P_i$,
and so is~$Q_j$.
\end{proof}

In view of the condition $G \backslash G_0 = D_1 \cup \ldots \cup
D_n$ and Proposition~\ref{non-divisible} the proof of the
irreducibility in $\mathbb C[G]$ of each function $Q_j$ reduces to
the following proposition.

\begin{proposition}
For every $j = 1, \ldots, m$ the restriction of\, $Q_j$ to the open
cell $G_0$ is an irreducible function in~$\mathbb C[G_0]$.
\end{proposition}

\begin{proof}
Let $j \in \{1, \ldots, m\}$. Then one has $\lambda_j = \omega_{j_1}
+ \ldots + \omega_{j_r}$. The restrictions to $G_0$ of the functions
$P_1, \ldots, P_n$ are invertible, therefore it suffices to prove
that the function $\widehat Q_j = Q_j/(P_{j_1} \ldots P_{j_r})$ is
irreducible in~$\mathbb C[G_0]$. We note that this function is
$(B\times H)$-semi-invariant of weight $(0, - \varphi_j)$. Since the
group $B \times H$ has an open orbit in~$G_0$, the subspace in
$\mathbb C[G_0]$ consisting of all $(B\times H)$-semi-invariant
functions of weight $(0, - \varphi_j)$ is one-dimensional. The proof
will be completed if we present an irreducible $(B\times
H)$-semi-invariant function of weight $(0, - \varphi_j)$ in~$\mathbb
C[G_0]$. The group $U \times N$ is connected and has no non-trivial
characters, hence in view of~\cite[Theorem~3.17]{VP} it suffices to
present an irreducible $(T\times S)$-semi-invariant function of
weight $(0, - \varphi_j)$ in~${}^U \mathbb C[G_0]^N$. One has ${}^U
\mathbb C[G_0]^N \simeq {}^U \mathbb C[U \times T\sigma_0 \times
U]^N \simeq \mathbb C[T \times U/N] \simeq \mathbb C[T \times
\mathfrak u/\mathfrak n]$, where the last isomorphism is due to an
$S$-equivariant isomorphism $U/N \simeq \mathfrak u/\mathfrak n$
(see~\cite[lemme~1.4]{Mon}). The group $T\times S$ acts on the
variety $T \times \mathfrak u/\mathfrak n$ as follows: any pair
$(t,s) \in T \times S$ takes each pair $(t_0, x) \in T \times
\mathfrak u/\mathfrak n$ to $(tt_0(\sigma_0 s^{-1} \sigma_0^{-1}),
sx)$. We recall that by Theorem~\ref{solvable_spherical} the
subspace $\mathfrak u/\mathfrak n$ is isomorphic as an $S$-module to
the direct sum $\mathbb C_{\varphi_1} \oplus \ldots \oplus \mathbb
C_{\varphi_m}$, where $\mathbb C_{\varphi_i} \simeq \mathbb C$ is
the subspace of weight $\varphi_i$ with respect to~$S$. Let $l_j$
denote the $j$th coordinate function on $\mathfrak u/\mathfrak n$.
Being a linear function, $l_j$ is an irreducible $S$-semi-invariant
function of weight~$-\varphi_j$. Hence the function $1 \otimes l_j
\in \mathbb C[T] \otimes \mathbb C[\mathfrak u/\mathfrak n] \simeq
\mathbb C[T \times \mathfrak u/ \mathfrak n]$ is an irreducible $(T
\times S)$-semi-invariant function of weight~$(0, -\varphi_j)$.
\end{proof}

\end{document}